\documentclass[10pt,twoside,reqno]{amsart}
\usepackage{amsmath}
\usepackage{amssymb}
\usepackage{amsfonts}
\usepackage{amsthm}
\usepackage{graphicx}
\usepackage[initials]{amsrefs}
\usepackage{fancyhdr}
\usepackage{hyperref}

\usepackage{tikz}
\usepackage{tikz-cd}

\usepackage[all,cmtip]{xy}

\newtheorem{theorem}{Theorem}[section]
\newtheorem{corollary}[theorem]{Corollary}
\newtheorem{lemma}[theorem]{Lemma}
\newtheorem{proposition}[theorem]{Proposition}
\theoremstyle{definition}
\newtheorem{definition}[theorem]{Definition}
\newtheorem{example}[theorem]{Example}
\newtheorem{remark}[theorem]{Remark}

\begin{document}

\title[Intrinsic Hilbert metrics on norms]{\large Intrinsic Hilbert metrics on cones of equivalent norms}
\author[Acosta-Portilla Juan Rafael]{Juan Rafael Acosta-Portilla$^{1}$}
\date{September 2026}
\maketitle

\begin{center}
{\footnotesize
$^{1}$Instituto de Investigaciones y Estudios Superiores Económicos y Sociales, Universidad Veracruzana, México\\
E-mail: juaacosta@uv.mx\\
}
\end{center}

\bigskip

{\footnotesize
\noindent
{\bf Abstract.}
Let $X$ be a Banach space and let $\mathcal{N}(X)$ denote the family of equivalent norms on $X$. We study Hilbert projective metrics on its projectivization $\mathcal{N}'(X)$ induced by ambient cones of nonnegative functions, with particular attention to the intrinsic Hilbert metric induced by the cone $\mathcal{N}(X)\cup\{0\}$. First, we show that the symmetric logarithmic metric on $\mathcal{N}'(X)$ is the Hilbert projective metric induced by the cone of nonnegative real-valued functions, and we characterize the ambient cones that induce the same metric. We then study the intrinsic order on $\mathcal{N}(X)$. For this purpose, we introduce the triangular defect
$\Delta_p(x,y)=p(x)+p(y)-p(x+y)$
and prove that two norms $p,q\in\mathcal{N}(X)$ belong to the same intrinsic part if and only if their triangular defects are uniformly comparable, that is,
$a\Delta_p\leq\Delta_q\leq b\Delta_p$
for some $a,b>0$. This yields an explicit formula for the intrinsic Hilbert metric in terms of the pointwise comparison of the norms and of their triangular defects. Finally, the map $\Phi(p)=(p,\Delta_p)$ realizes the intrinsic cone order of $\mathcal{N}(X)$ inside a canonical product cone of nonnegative functions and preserves the Hilbert metric on each intrinsic part.


}

\bigskip

\section{Introduction}

Banach space geometry studies, among other phenomena, geometric properties related to the rotundity and smoothness of the unit ball \cite{Banach1932, clarkson1936uniformly, Mazur1933, Smulyan1940}, the asymptotic behavior of weakly null sequences \cite{prus1992opial, prus1997noncreasy, prus2005nearly}, packing properties of the unit ball \cite{james1964uniformly, kottman1970packing} and fixed point properties for nonexpansive mappings \cite{betiuk2013susuki, agarwal2009fixedpointbook, goebel1990topics}. One of the main tools for studying these phenomena, and for constructing examples and counterexamples exhibiting particular geometric behaviors, is renorming theory \cite{DevilleGodefroyZizler1993SmoothnessRenormings, Godefroy2001, GuiraoMontesinosAizler2022RenormingsToolbox}. From this point of view, however, the emphasis is usually placed on individual norms and on the geometric properties they induce.

More precisely, let $X$ be a fixed Banach space and let $\mathcal{N}(X)$ denote the family of all equivalent norms on $X$. A large part of renorming theory is concerned with the existence, construction, or characterization of particular elements of $\mathcal{N}(X)$ satisfying prescribed geometric properties. In this sense, the objects under study are usually the individual elements of $\mathcal{N}(X)$ rather than the structure of $\mathcal{N}(X)$ as a whole.

A complementary point of view is to regard the family of equivalent norms itself as a mathematical object and to investigate its global structure. Several works have considered metrics and topological structures on spaces of equivalent norms, arising for instance from distortion, renorming, or genericity considerations \cite{TomczakJaegermann1989BanachMazurDA, dominguezPhoti2008porosity, dominguezPhoti2010genericityinsomebanach, fabianZajicekZizler1982residualityRotundNorms}. This suggests a broader problem: to understand which geometric, metric, and algebraic structures arise naturally on $\mathcal{N}(X)$ itself.

Hilbert geometry provides a natural framework for addressing this question. The classical Hilbert metric originates from the cross-ratio geometry of convex domains and admits a projective formulation on positive cones \cite{Hilbert1895, Bushell1973hilbert'sMetric}. This point of view was developed in the work of Birkhoff and Samelson in connection with positive operators and the Perron--Frobenius theory \cite{Birkhoff1957, Samelson1957PerronFrobeniustheorem}, and was later formulated in the setting of cones in Banach spaces by Bushell \cite{Bushell1973hilbert'sMetric}. Subsequent developments have shown that Hilbert's projective metric provides a useful link between cone-induced orders, positive operators, and nonlinear analysis \cite{LemmensNussbaum2012, LemmensNussbaum2014, nussbaum1988hilbert, nussbaum1989hilbert}.

In the present work, we adopt an intrinsic point of view. Rather than imposing an external geometric structure on $\mathcal{N}(X)$, we begin with structures that are already present in the family of equivalent norms. The set $\mathcal{N}(X)\cup\{0\}$ is naturally a cone: positive scalar multiples and sums of equivalent norms remain equivalent norms. Consequently, this cone carries an intrinsic order, and its projectivization

\begin{equation}
\mathcal{N}'(X)=[\mathcal{N}(X)]
\end{equation}

\noindent inherits the corresponding Hilbert projective metric. Our purpose is to study this intrinsic metric structure and to compare it with other natural metrics already considered on the projectivized space of equivalent norms. In particular, we are interested in the symmetric logarithmic metric, which is naturally defined in terms of the optimal equivalence constants between norms:

\begin{equation}
d_{\log}([p],[q]) = \log \frac{u}{l}
\end{equation}

\noindent for every $p , q \in \mathcal{N}(X)$, where $u \geq l >0$ are the optimal constants such that

\begin{equation}
l p (x) \leq q(x) \leq u p(x)
\end{equation}

\noindent for every $x \in X$. This leads us to the following question: 

\begin{center}
{\it what is the Hilbert geometry induced by the cone $\mathcal{N}(X)\cup\{0\}$ itself, and how is it related to the symmetric logarithmic metric $d_{\log}$?}
\end{center}

As a first step, we identify $d_{\log}$ with the Hilbert projective metric induced on $\mathcal{N}'(X)$ by the ambient cone $\mathbb{R}^X_+$ of nonnegative real-valued functions from $X$ to $\mathbb{R}$. We then consider more general cones containing subfamilies of $\mathcal{N}(X)$ and characterize those for which the induced Hilbert projective metric coincides with $d_{\log}$. 

We then turn to the intrinsic Hilbert metric induced by the cone $\mathcal{N}(X)\cup \{0\}$. This case presents an additional obstruction. Pointwise nonnegativity of a difference $p-q$ does not guarantee that it is an equivalent norm: it may fail to be positive away from the origin in a uniform sense, and it may fail to satisfy the triangle inequality. To encode the latter obstruction, for every $p\in\mathcal{N}(X)$ we introduce the triangular defect

\begin{equation}
\Delta_p(x,y)=p(x)+p(y)-p(x+y).
\end{equation}

\noindent Using this function, we characterize the intrinsic parts of $\mathcal{N}(X)$ in terms of uniform comparison between triangular defects. We then obtain an explicit formula for the intrinsic Hilbert metric in terms of two comparison mechanisms: the pointwise comparison between the norms and the corresponding comparison between their triangular defects.

\begin{equation}
d_{\mathcal{N}(X)}([p],[q]) = \log \frac{ \max \{ M_{\log}(q/p), M_{\log}(\Delta_q/\Delta_p)\}}{\min \{ m_{\log}(q/p), m_{\log}(\Delta_q/\Delta_p) \}}.
\end{equation}
 
Finally, we show that this decomposition is not merely formal. The map

\begin{equation}
\Phi(p)=(p,\Delta_p)
\end{equation}

\noindent realizes the intrinsic geometry of $\mathcal{N}(X)$ inside a product cone of nonnegative functions. After identifying the appropriate cone in the first coordinate, $\Phi$ preserves both the intrinsic cone order and the corresponding Hilbert metric on each intrinsic part. Thus, the intrinsic geometry of $\mathcal{N}(X)$ admits a natural representation in terms of two coordinates: the first records the pointwise comparison of norms, while the second records the pointwise comparison of their triangular defects.

Throughout the paper, we provide examples illustrating limiting cases and showing that the distinctions between the different notions considered are genuine.

\section{Preliminaries}
In this section, we introduce the basic definitions and fix the notation used throughout the paper. The section is divided into four subsections, devoted respectively to cones and cone-induced orders, parts and projectivization, the Hilbert projective metric, and the symmetric logarithmic metric on $\mathcal{N}'(X)$.

\subsection*{Cones and cone-induced orders} 
We begin by introducing cones, the orders induced by them, and a basic structural property that will be used throughout the paper. The definitions used in this subsection are standard and can be found in most texts on convex analysis. However, because of their close connection with the topics considered in this work, we refer the reader in particular to \cite{LemmensNussbaum2012, LemmensNussbaum2014}.

\begin{definition}[Positive cone]
Let $X$ be a real vector space. A nonempty convex $C\subset X$ is called a cone if

\begin{equation}
\lambda C \subset C
\end{equation}

\noindent for every $\lambda > 0$  and 

\begin{equation}
C \cap (- C) \subset \{ 0 \}
\end{equation}

\noindent If $0 \in C$ we say that $C$ is a cone with vertex at $0$ or a pointed cone.
\end{definition}

\begin{definition}[Order induced by a cone]
Let $C$ be a cone in a real vector space $X$. The relation $\preceq_C$ on $X$ induced by $C$ is defined by

\begin{equation}
p\preceq_C q  \quad \text{if} \quad q - p \in C.
\end{equation}

\noindent This relation is transitive. Moreover, if $C$ has a vertex, then it is reflexive and antisymmetric, and hence it is a partial order.
\end{definition}

%
%

\begin{definition}[Almost Archimedean cone]
Let $C$ be a cone with vertex in a real vector space $X$. The cone $C$ is said to be \emph{almost Archimedean} if, for every $p\in X$ and $u\in C$, the inequalities

\begin{equation}
-\varepsilon u\preceq_C p\preceq_C\varepsilon u
\end{equation}

\noindent for every $\varepsilon>0$ imply that $p=0$.
\end{definition}



\subsection*{Parts and projectivization} 

\begin{definition}[Parts of a cone]
Let $C$ be a cone. Two nonzero elements $p,q\in C$ are said to be comparable if there exist constants $\alpha,\beta>0$ such that

\begin{equation}
\alpha p \preceq_C q \preceq_C\beta p.
\end{equation}

\noindent This relation is an equivalence relation on $C\setminus \{0\}$, and its equivalence classes are called the parts of $C$. If $p \in C \setminus \{0 \}$ it is clear that $\alpha p \preceq_C p \preceq_C \beta p$ for any $0 < \alpha < 1 < \beta$. The part containing $p$ is

\begin{equation}
C_p = \{q \in C\setminus\{0 \} \mid q \text{ is comparable to } p \}.
\end{equation}

\noindent Moreover, $C_p$ is a cone without vertex.
\end{definition}

\begin{definition}[Projectivization of a cone]
Given a cone $C$, the projectivization of $C$ is 

\begin{equation}
[C] = (C \setminus \{ 0\}) / \sim
\end{equation}

\noindent where $\sim$ is the equivalence relation given by $p \sim q$ if there exists $\lambda > 0$ such that $q = \lambda p$. For every $D \subset C \setminus \{0\}$, we denote by $[D] \subset [C]$ its projectivization, namely

\begin{equation}
[D] = \{[p] \in [C] \mid p \in D \}.
\end{equation} 
\end{definition}

\subsection*{Hilbert projective metric}
In this subsection we present the Hilbert projective metric. The following definition can be found in \cite{Bushell1973hilbert'sMetric, LemmensNussbaum2012, nussbaum1988hilbert, nussbaum1989hilbert, PapadopoulosTroyanov2014HandBookHilbertGeometry}

\begin{definition}[Hilbert's metric]
Let $C$ be a cone. For every comparable $p , q \in C \setminus\{ 0\}$ set

\begin{equation}
m_C(q/p) = \sup \{ \alpha > 0 \mid \alpha p \preceq_C q\} \qquad M_C(q/p) = \inf \{\beta > 0 \mid q \preceq_C \beta p \}.
\end{equation}

\noindent Note that $0 < m_C(q/p) \leq  M_C(q/p) < \infty$. For comparable elements $p , q \in C\setminus \{0\}$ we define the Hilbert projective metric by

\begin{equation}
d_C(p , q) = \log \frac{M_C(q/p)}{m_C(q/p)}.
\end{equation}

\noindent If $p$ and $q$ are two nonzero elements of $C$ which are not comparable, we define $d_C(p,q) = \infty$. Let $p \in C\setminus \{0 \}$ and let $C_p$ be the part containing $p$. If $C$ is an almost Archimedean cone, then the Hilbert projective distance induces a metric on the projectivized part $[C_p]$, defined by

\begin{equation}
d_C([p], [q]) = d_C(p , q).
\end{equation}

\noindent Moreover, the almost Archimedean property is equivalent to the fact that $d_C$ is a metric on every part of $C$.
\end{definition}

\begin{proposition}\label{Proposition comparison of Hilbert seminorms wrt subsets}
Let $C$ and $D$ be two cones with vertex such that $C \subset D$. Then, for all $C$-comparable elements $p,q\in C$, one has

\begin{equation}
M_C(q/p) \geq M_D(q/p), \qquad m_C(q/p) \leq m_D(q/p).
\end{equation}

\noindent Thus, if $p$ and $q$ belong to the same part of $C$, then they also belong to the same part of $D$, and

\begin{equation}
d_C([p],[q])\geq d_D([p],[q]).
\end{equation}

\noindent Moreover, if 

\begin{equation}
d_C([p],[q]) = d_D([p],[q]),
\end{equation} 

\noindent then 

\begin{equation}
M_C(q/p)=M_D(q/p), \qquad m_C(q/p)=m_D(q/p).
\end{equation}
\end{proposition}

\noindent Thus, smaller cones induce smaller parts and larger Hilbert distances. Moreover, equality of Hilbert distances under a cone inclusion is rigid: it forces equality of the upper and lower comparison constants separately, not only equality of their quotient.


%
%
%
%

\subsection*{Symmetric logarithmic metric on equivalent norms}

Several metrics on the space of equivalent norms $\mathcal N(X)$ have been considered in the literature; see, for instance, \cite{dominguezPhoti2008porosity, dominguezPhoti2010genericityinsomebanach, fabianZajicekZizler1982residualityRotundNorms, TomczakJaegermann1989BanachMazurDA}. In this subsection, we present the symmetric logarithmic distance in $\mathcal{N}(X)$.

Let $X$ be a Banach space. We denote by $\mathcal{N}(X)$ the cone of equivalent norms on $X$ and $\mathcal{N}'(X) := [\mathcal{N}(X)]$ its projectivization. For each $p, q \in \mathcal{N}(X)$ there exist optimal $u \geq l > 0$ such that

\begin{equation}\label{optimal constants for equivalent norms}
l p(x) \leq q(x) \leq u p(x) 
\end{equation} 

\noindent for each $x \in X$. It is clear that the values

\begin{equation}\label{definition of symmetric log values}
M_{\log}(q/p) := \sup_{x\neq 0}\frac{q(x)}{p(x)}, \qquad m_{\log}(q/p) := \inf_{x\neq 0}\frac{q(x)}{p(x)}
\end{equation}

\noindent satisfy 

\begin{equation}
M_{\log}(q/p) = u , \qquad m_{\log}(q/p) = l.
\end{equation}

A detailed exposition of the following definition can be found in \cite{dominguezPhoti2008porosity, dominguezPhoti2010genericityinsomebanach}

\begin{definition}[Log symmetric distance]
The symmetric logarithmic metric $d$ on $N'(X)$ is defined by

\begin{equation}\label{definition distance d}
d_{\log} \left([p], [q] \right) = \log \left( \frac{u}{l} \right) =\log \frac{M_{\log}(q/p)}{m_{\log}(q/p)}
\end{equation}

\noindent for each $[p], [q] \in \mathcal{N}'(X)$.
\end{definition}

\noindent We denote by

\begin{equation}
\mathbb{R}^X_+=\{f:X\to\mathbb{R}\mid f(x)\geq 0\text{ for every }x\in X\},
\end{equation} 

\noindent the cone of nonnegative real-valued functions. In the following proposition, we relate the symmetric logarithmic distance to the Hilbert projective metric.

\begin{proposition}\label{proposition pointwise cone induces logarithmic metric}
Let $X$ be a Banach space, and let $p,q\in\mathcal{N}(X)$. If

\begin{equation}
C=\mathbb{R}^X_+,
\end{equation}

\noindent then

\begin{equation}
M_C(q/p)=M_{\log}(q/p), \qquad m_C(q/p)=m_{\log}(q/p).
\end{equation}

\noindent Consequently, the Hilbert projective metric induced by $\mathbb{R}^X_+$ on $\mathcal{N}'(X)$ coincides with the symmetric logarithmic metric; that is,

\begin{equation}
d_C([p],[q])=d_{\log}([p],[q])
\end{equation}

\noindent for every $p,q\in\mathcal{N}(X)$.
\end{proposition}

\section{Ambient cone metrics on $\mathcal{N}'(X)$}

In this section, we relate the symmetric logarithmic distance on $\mathcal{N}'(X)$ to Hilbert projective metrics induced by ambient cones containing $\mathcal{N}(X)$.

The cone $\mathcal{N}(X)$ admits several natural ambient structures. It is contained in the cone $\mathcal{S}(X)$ of continuous seminorms on $X$ and also in the cone of nonnegative real-valued functions $\mathbb{R}^X_+$, as well as in the vector spaces $\operatorname{Lip}(X,\mathbb{R})$ of Lipschitz functions and $C(X,\mathbb{R})$ of continuous functions. It also generates the natural algebraic ambient space

\begin{equation}
\operatorname{span}\mathcal{N}(X) = \mathcal{N}(X)-\mathcal{N}(X),
\end{equation}

\noindent whose elements are differences of equivalent norms on $X$. However, not all of these ambient spaces generate the symmetric logarithmic metric as a Hilbert projective metric, as shown in Theorem \ref{theorem characterization of log metric as Hilbert metric}.

\begin{proposition}
Let $X$ be a set. Then $\mathbb{R}^X_+$ is an almost Archimedean cone. In particular, every subcone of $\mathbb{R}^X_+$ with vertex is almost Archimedean. Hence every cone of nonnegative real-valued functions with vertex induces a genuine Hilbert projective metric on each of its parts.
\end{proposition}

\begin{theorem}\label{theorem characterization of log metric as Hilbert metric}
Let $X$ be a Banach space, let $C$ be a subcone of $\mathcal{N}(X)$, and let $D\subseteq \mathbb{R}^{X}_{+}$ be a cone with vertex such that $C\subseteq D$. Then the following statements are equivalent:

\begin{itemize}
\item[1)] The cone $C$ is contained in a single part of $D$, and the Hilbert projective metric induced by $D$ on $[C]$ coincides with the symmetric logarithmic metric; that is,

\begin{equation}
d_D([p],[q]) = \log\frac{M_{\log}(q/p)}{m_{\log}(q/p)}
\end{equation}

\noindent for every $p,q\in C$.

\item[2)] For every $p,q\in C$, one has
\begin{equation}
\beta p-q\in D
\end{equation}

\noindent whenever $\beta>M_{\log}(q/p)$, and

\begin{equation}
q-\alpha p\in D
\end{equation}

\noindent whenever $0<\alpha<m_{\log}(q/p)$.

\item[3)] For every $p,q\in C$ and $0 < \theta <1$ such that

\begin{equation}
q(x)\leq \theta p(x)
\end{equation}

\noindent for every $x\in X$, we have

\begin{equation}
p-q\in D.
\end{equation}

\item[(4)] The cone $D$ contains every uniformly positive difference of elements of $C$; that is,

\begin{equation}
\left\{
f\in C-C \;\middle|\;
\begin{array}{c}
\text{there exist }r\in C\text{ and }\varepsilon>0\text{ such that}\\[1mm]
f(x)\geq \varepsilon r(x)\text{ for every }x\in X
\end{array}
\right\}
\subset D.
\end{equation}
\end{itemize}
\end{theorem}

\begin{proof}
$1)$ implies $2)$. Let $p , q \in C$. Since $D \subset \mathbb{R}^X_+$ and the metrics coincide, by Proposition \ref{Proposition comparison of Hilbert seminorms wrt subsets} we have 

\begin{equation}
M_D(q/p) = M_{\log}(q/p), \qquad m_D(q/p) = m_{\log}(q/p).
\end{equation}

\noindent Let $\beta > M_{\log}(q/p)$ and $\alpha < m_{\log}(q /p)$. Since $M_D(q/p)= M_{\log}(q/p)$ is an infimum, there exists $\beta > \beta' \geq M_D(q/p)$ such that 

\begin{equation}
\beta' p - q \in D.
\end{equation}

\noindent Thus, $(\beta - \beta') p \in D$ implies 

\begin{equation}
\beta p - q = \beta' p - q + (\beta -\beta') p \in  D.
\end{equation}

\noindent A similar argument proves that $q - \alpha p  \in D$.

$2)$ implies $3)$. Let $p , q \in C$ and $0 < \theta < 1$ such that $q(x ) \leq \theta p(x)$ for every $x \in X$. By the definition of $M_{\log}(q/p)$ in Equation (\ref{definition of symmetric log values}), we have $M_{\log}(q/p) \leq \theta < 1$. Setting $\beta= 1$, we obtain $\beta > M_{\log}(q/p)$ and

\begin{equation}
p- q = \beta p - q \in D.
\end{equation}

$3)$ implies $4)$. Let $p,q, r \in C$ and $\varepsilon > 0$ such that

\begin{equation}
p(x) - q(x) \geq \varepsilon r(x)
\end{equation}

\noindent for each $x \in X$. In particular 

\begin{equation}\label{technicall positive difference}
p(x_0) - q(x_0) \geq \varepsilon r(x_0) > 0
\end{equation}

\noindent for any $x_0 \neq 0$. Since $p$ and $r$ are equivalent norms, there exists $ \theta'> 0$ such that

\begin{equation}
\theta' p(x) \leq \varepsilon r(x) 
\end{equation} 

\noindent for each $x \in X$. Since $q$ is nonnegative, $ p(x) - q(x) \leq p(x)$ for every $x \in X$. Hence   

\begin{equation}\label{technicall less than one difference}
\theta' p(x) \leq  p(x) - q(x) \leq p(x).
\end{equation}

\noindent Thus, (\ref{technicall positive difference})  and (\ref{technicall less than one difference}) imply $0 < \theta'< 1$ and 

\begin{equation}
q(x) \leq (1 - \theta')p(x) 
\end{equation}

\noindent for each $x \in X$. Then

\begin{equation}
p - q \in D.
\end{equation} 

$4)$ implies $1)$. Let $p,  q \in C$ and $\varepsilon > 0$. Since $p$ and $q$ are equivalent norms, there exist optimal 

\begin{equation}
\alpha = m_{\log}(q/p), \qquad \beta = M_{\log}(q/p)
\end{equation}
 
\noindent such that

\begin{equation}
\alpha p(x) \leq q(x) \leq \beta p(x)
\end{equation}

\noindent for each $x \in X$. Hence

\begin{equation}
\begin{split}
\beta p(x) - q(x) & \geq 0 \\
(\beta + \varepsilon) p(x) - q(x)  & \geq \varepsilon p(x)
\end{split}
\end{equation}

\noindent for each $x\in X$. Thus 

\begin{equation}\label{technical eq for each varepsilon comparison 1}
(\beta + \varepsilon) p - q \in D
\end{equation}

\noindent By a similar argument

\begin{equation}
(1 + \varepsilon)q(x) - \alpha p(x) \geq \varepsilon q (x)
\end{equation} 

\noindent for each $x \in X$ and

\begin{equation}\label{technical eq for each varepsilon comparison 2}
q - (1 + \varepsilon)^{-1} \alpha p \in D
\end{equation}

\noindent Then $p$ and $q$ are comparable under the cone order $\preceq_D$. Finally since (\ref{technical eq for each varepsilon comparison 1}) and (\ref{technical eq for each varepsilon comparison 2}) are valid for each $\varepsilon > 0$ it follows that

\begin{equation}
M_D(q/p) \leq M_{\log}(q/p) , \qquad m_D(q/p) \geq m_{\log}(q/p).
\end{equation}

\noindent That is, 

\begin{equation}
d_D([p],[q]) \leq d_{\log}([p],[q]).
\end{equation}

\noindent Using Proposition \ref{Proposition comparison of Hilbert seminorms wrt subsets} we conclude that

\begin{equation}
d_D([p] , [q]) = d_{\log}([p], [q])
\end{equation} 
\end{proof}

\begin{remark}
Let $X$ be a Banach space.

\begin{itemize}
\item[i)] The Hilbert projective metric induced by the cone $\mathcal{S}(X)$ of continuous seminorms does not coincide globally on $\mathcal{N}'(X)$ with the symmetric logarithmic metric. Indeed, two equivalent norms need not belong to the same part of $\mathcal{S}(X)$, since the pointwise comparison of two norms does not imply that their difference is a seminorm.

\item[ii)] The same phenomenon occurs for the intrinsic cone $\mathcal{N}(X) \cup \{0 \}$. In general, $\mathcal{N}(X)$ is not a single part with respect to its intrinsic order. Consequently, the Hilbert projective metric induced by $\mathcal{N}(X)$ is defined only on its parts and does not coincide globally on $\mathcal{N}'(X)$ with the symmetric logarithmic metric.

\item[iii)] Every equivalent norm on $X$ is a continuous Lipschitz function. Therefore, the Hilbert projective metrics induced on $\mathcal{N}'(X)$ by the positive cones of $C(X,\mathbb{R})$ and $\operatorname{Lip}(X,\mathbb{R})$ coincide with the symmetric logarithmic metric.
\end{itemize}
\end{remark}

\begin{example}
We present an example of two equivalent norms that do not belong to the same part of $\mathcal{S}(X)$ and, consequently, neither to the same part of $\mathcal{N}(X)$. 

Consider the following norms on $X = \mathbb{R}^{2}$.

\begin{equation}
p(a,b)=\|(a,b)\|_{1}= |a| + |b|, \qquad q(a,b)=\| (a,b)\|_{\infty} = \max \{|a|, |b| \}.
\end{equation}

For every $\alpha>0$, the function $q-\alpha p$ is not a seminorm. Indeed, taking $x=(1,1)$ and $y=(1,-1)$, one has

\begin{equation}
(q-\alpha p)(x+y)=2-2\alpha
\end{equation}

\noindent and

\begin{equation}
(q-\alpha p)(x)+(q-\alpha p)(y)=2-4\alpha.
\end{equation}

\noindent Hence

\begin{equation}
(q-\alpha p)(x+y) = 2 - 2 \alpha > 2 - 4 \alpha = (q-\alpha p)(x)+(q-\alpha p)(y),
\end{equation}

\noindent so the triangle inequality fails. Therefore, there is no $\alpha>0$ such that $\alpha p\preceq_{\mathcal{S}} q$. Thus, $p$ and $q$ belong to different parts of the cone of seminorms on $\mathbb{R}^2$, $\mathcal{S}(\mathbb{R}^{2})$.
\end{example}

\section{Intrinsic cone order on $\mathcal{N}(X)$}

In this section, we study the Hilbert metric induced by $\mathcal{N}(X)$ on its projectivization $\mathcal{N}'(X)$. For this purpose, we introduce the following definition.

\begin{definition}[Intrinsic Hilbert metric]
Let $C$ be a cone and set

\begin{equation}
C_0=C\cup\{0\}.
\end{equation}

\noindent We call the Hilbert projective metric induced by the cone order $\preceq_{C_0}$ on each projectivized part of $C$ the intrinsic Hilbert metric of $C$.
\end{definition}

As observed in the previous section, the intrinsic Hilbert metric on $\mathcal{N}(X)$ does not, in general, coincide with the symmetric logarithmic metric. The obstruction arises from the fact that differences of equivalent norms need not remain in $\mathcal{N}(X)$. In order to express this phenomenon in geometric terms and to describe the intrinsic Hilbert metric on $\mathcal{N}(X)$, we introduce the following notion.

\begin{definition}[Triangular defect]
Let $(X,p)$ be a Banach space. The triangular defect associated with $p$ is the function

\begin{equation}
\Delta_p:X\times X\longrightarrow\mathbb{R}
\end{equation}

\noindent defined by

\begin{equation}
\Delta_p(x,y)=p(x)+p(y)-p(x+y).
\end{equation}
\end{definition}

The triangular defect measures the amount by which the triangle inequality fails to be an equality. In particular, $\Delta_p(x,y)\geq 0$ for every $x,y\in X$, and $\Delta_p(x,y)=0$ precisely when equality holds in the triangle inequality for the pair $(x,y)$.

\begin{proposition}\label{Proposition properties of triangular defect}
Let $X$ be a Banach space and $p , q \in  \mathcal{N}(X)$. Then the triangular defect satisfies the following properties:

\begin{itemize}
\item[1)] For every $x,y\in X$ and $\lambda>0$,

\begin{equation}
\Delta_p(\lambda x,\lambda y)=\lambda\Delta_p(x,y) = \Delta_{\lambda p}(x,y)
\end{equation}

\noindent and

\begin{equation}
\Delta_{p + q}(x , y) = \Delta_p(x,y) + \Delta_q(x, y).
\end{equation}

\item[2)] For every nonzero $x,y\in X$,
\begin{equation}
\Delta_p(x,y)=0
\end{equation}

\noindent if and only if

\begin{equation}
\Delta_p\left(\frac{x}{p(x)},\frac{y}{p(y)}\right)=0.
\end{equation}

\item[3)] For every nonzero $x,y\in X$ and every $\alpha,\beta>0$,

\begin{equation}
\Delta_p(x,y)>0
\end{equation}

\noindent if and only if

\begin{equation}
\Delta_p(\alpha x,\beta y)>0.
\end{equation}
\end{itemize}
\end{proposition}
\begin{proof}
$1)$ follows directly from the definition.

For $2)$, the collinear cases are immediate. Let $x,y\in X\setminus \{0 \}$ be noncollinear. Set $a=p(x)$, $b=p(y)$, $u=x/a$, $v=y/b$, and $t=a/(a+b)$. Suppose first that $\Delta_p(x,y)=0$. Then $p(au+bv)=a+b$, and hence $p(tu+(1-t)v)=1$. It is enough to prove that $p((u+v)/2)=1$. By convexity, $p((u+v)/2)\leq 1$.

\noindent If $t\leq 1/2$, then

\begin{equation}
tu+(1-t)v=2t\frac{u+v}{2}+(1-2t)v.
\end{equation}

\noindent Therefore,

\begin{equation}
1\leq 2t p\left(\frac{u+v}{2}\right)+(1-2t),
\end{equation}

\noindent and consequently $p((u+v)/2)\geq 1$. If $t\geq 1/2$, then

\begin{equation}
tu+(1-t)v=2(1-t)\frac{u+v}{2}+(2t-1)u,
\end{equation}

\noindent  and the same argument gives $p((u+v)/2)\geq 1$. Thus $p((u+v)/2)=1$, and therefore $\Delta_p(u,v)=0$.

Conversely, suppose that $\Delta_p(u,v)=0$, equivalently $p((u+v)/2)=1$. Since $tu+(1-t)v$ is a convex combination of $u$ and $v$, one has $p(tu+(1-t)v)\leq 1$. If $t\leq 1/2$, then

\begin{equation}
\frac{u+v}{2}=\frac{1}{2(1-t)}\big(tu+(1-t)v\big)+\frac{1-2t}{2(1-t)}u,
\end{equation}

\noindent and hence $p(tu+(1-t)v)\geq 1$. If $t\geq 1/2$, then

\begin{equation}
\frac{u+v}{2}=\frac{1}{2t}\big(tu+(1-t)v\big)+\frac{2t-1}{2t}v,
\end{equation}

\noindent and again $p(tu+(1-t)v)\geq 1$. Therefore $p(tu+(1-t)v)=1$, and

\begin{equation}
p(x+y)=p(au+bv)=(a+b)p(tu+(1-t)v)=a+b=p(x)+p(y).
\end{equation}

\noindent Thus $\Delta_p(x,y)=0$.

Finally, $3)$ follows from $2)$. Indeed, for $\alpha,\beta>0$,

\begin{equation}
\frac{\alpha x}{p(\alpha x)}=\frac{x}{p(x)}, \qquad \frac{\beta y}{p(\beta y)}=\frac{y}{p(y)}.
\end{equation}

\noindent Hence $\Delta_p(\alpha x,\beta y)=0$ if and only if $\Delta_p(x,y)=0$. Since the triangular defect is always nonnegative, this is equivalent to

\begin{equation}
\Delta_p(\alpha x,\beta y)>0
\end{equation}

\noindent if and only if $\Delta_p(x,y)>0$.
\end{proof}

\begin{remark}\label{remark properties of triangular defect}
The previous proposition gives a geometric interpretation of the triangular defect.

\begin{itemize}
\item[i)] By Item $2)$ of Proposition \ref{Proposition properties of triangular defect}, for nonzero $x,y\in X$, the condition $\Delta_p(x,y)=0$ depends only on the normalized vectors $x/p(x)$ and $y/p(y)$. More precisely,

\begin{equation}
\Delta_p(x,y)=0
\end{equation}

\noindent if and only if the segment

\begin{equation}
\left[\frac{x}{p(x)},\frac{y}{p(y)}\right]
\end{equation}

\noindent is contained in the unit sphere of $(X,p)$. Equivalently, there exists a supporting functional $x^*\in X^*$ with dual norm one with respect to $p$, $\|x^*\|_{p*}=1$, such that

\begin{equation}
x^*\left(\frac{x}{p(x)}\right)=x^*\left(\frac{y}{p(y)}\right)=1.
\end{equation}

\noindent Thus, $\Delta_p(x,y)>0$ precisely when the portion of the unit sphere joining the two normalized directions is not flat.

\item[ii)] By Item $3)$ of Proposition \ref{Proposition properties of triangular defect}, the qualitative behavior of the triangular defect is invariant under independent positive rescaling of its arguments. In particular,

\begin{equation}
\Delta_p(x,y)>0
\end{equation}

\noindent if and only if

\begin{equation}
\Delta_p(\alpha x,\beta y)>0
\end{equation}

\noindent for every $\alpha,\beta>0$. Hence, from a qualitative point of view, the triangular defect is a property of the positive rays generated by $x$ and $y$, rather than of the particular representatives chosen on those rays.

\item[iii)] Moreover, for nonzero $x,y\in X$, the condition $\Delta_p(x,y)=0$ is equivalent to the linearity of $p$ on the cone generated by $x$ and $y$. More precisely,
\begin{equation}
\Delta_p(x,y)=0
\end{equation}

\noindent if and only if, for every $a,b>0$,

\begin{equation}
p(ax+by)=a p(x)+b p(y).
\end{equation}

\noindent Equivalently, in terms of the normalized vectors,

\begin{equation}
p\left(a\frac{x}{p(x)}+b\frac{y}{p(y)}\right)=a+b
\end{equation}

\noindent for every $a,b>0$. Thus, the vanishing of the triangular defect is an intrinsic property of the cone generated by $x$ and $y$.
\end{itemize}
\end{remark}

The following technical lemma characterizes when a difference of two equivalent norms is again an equivalent norm.

\begin{lemma}\label{lemma characterization difference norms triangular defect}
Let $X$ be a Banach space, $p,q\in\mathcal{N}(X)$, and $\lambda,\theta>0$. Then the following statements are equivalent:

\begin{itemize}
\item[1)] $\lambda p-\theta q\in\mathcal{N}(X)$.

\item[2)] There exists $\varepsilon>0$ such that

\begin{equation}
\lambda p(x)-\theta q(x)\geq\varepsilon p(x)
\end{equation}

\noindent for every $x\in X$, and

\begin{equation}
\lambda\Delta_p(x,y)\geq\theta\Delta_q(x,y)
\end{equation}

\noindent for every $x,y\in X$.
\end{itemize}
\end{lemma}
\begin{proof}
$1)$ implies $2)$. Since $\lambda p-\theta q\in\mathcal{N}(X)$, there exists $\varepsilon>0$ such that

\begin{equation}
\lambda p(x)-\theta q(x)\geq\varepsilon p(x)
\end{equation}

\noindent for every $x\in X$. Moreover, $\lambda p-\theta q$ satisfies the triangle inequality, and therefore

\begin{equation}
\lambda p(x+y)-\theta q(x+y)\leq \lambda p(x)-\theta q(x)+\lambda p(y)-\theta q(y).
\end{equation}

\noindent Equivalently,

\begin{equation}
\begin{split}
\lambda\big(p(x)+p(y)-p(x+y)\big) &\geq\theta\big(q(x)+q(y)-q(x+y)\big) \\
\lambda\Delta_p(x,y)&\geq\theta\Delta_q(x,y).
\end{split}
\end{equation}

$2)$ implies $1)$. Set $r=\lambda p-\theta q$. By assumption, there exists $\varepsilon>0$ such that
\begin{equation}
r(x)\geq\varepsilon p(x)
\end{equation}

\noindent for every $x\in X$. Hence $r(x)=0$ if and only if $x=0$. Moreover, $r$ is absolutely homogeneous since $p,q\in\mathcal{N}(X)$, and

\begin{equation}
\varepsilon p(x)\leq r(x)\leq\lambda p(x)
\end{equation}

\noindent for every $x \in X$. Thus, it remains only to prove the triangle inequality. By hypothesis,

\begin{equation}
\Delta_r(x,y)=\lambda\Delta_p(x,y)-\theta\Delta_q(x,y)\geq0
\end{equation}

\noindent for every $x,y\in X$. Therefore,

\begin{equation}
r(x+y)\leq r(x)+r(y).
\end{equation}

\noindent Consequently, $\lambda p-\theta q = r \in\mathcal{N}(X)$.
\end{proof}

The following theorem characterizes the parts of the intrinsic order of $\mathcal{N}(X)$ in terms of uniform comparison between triangular defects.

\begin{theorem}\label{theorem intrinsic parts triangular defect}
Let $X$ be a Banach space and let $p,q\in\mathcal{N}(X)$. Then the following statements are equivalent:

\begin{itemize}
\item[1)] The norms $p$ and $q$ belong to the same intrinsic part of $\mathcal{N}(X)$.

\item[2)] There exist $a,b>0$ such that

\begin{equation}
a\Delta_p(x,y)\leq\Delta_q(x,y)\leq b\Delta_p(x,y)
\end{equation}

\noindent for every $x,y\in X$.
\end{itemize}
\end{theorem}
\begin{proof}
$1)$ implies $2)$. Since $p$ and $q$ belong to the same intrinsic part of $\mathcal{N}(X)$, there exist $a,b>0$ such that

\begin{equation}
bp-q\in\mathcal{N}(X), \qquad q-ap\in\mathcal{N}(X).
\end{equation}

\noindent By the triangle inequality applied to both norms and Lemma \ref{lemma characterization difference norms triangular defect}, it follows that

\begin{equation}
b\Delta_p(x,y)\geq\Delta_q(x,y), \qquad \Delta_q(x,y)\geq a\Delta_p(x,y)
\end{equation}

\noindent for every $x,y\in X$. Hence

\begin{equation}
a\Delta_p(x,y)\leq\Delta_q(x,y)\leq b\Delta_p(x,y)
\end{equation}

\noindent for every $x,y\in X$.

$2)$ implies $1)$. Since $p$ and $q$ are equivalent norms, there exist $u\geq l>0$ such that

\begin{equation}
lp(x)\leq q(x)\leq up(x)
\end{equation}

\noindent for every $x\in X$. By hypothesis, there exist $a,b>0$ such that

\begin{equation}
a\Delta_p(x,y)\leq\Delta_q(x,y)\leq b\Delta_p(x,y)
\end{equation}

\noindent for every $x,y\in X$. Choose $\beta>\max\{u,b\}$. Then

\begin{equation}
\beta p(x)-q(x)\geq(\beta-u)p(x)>0
\end{equation}

\noindent for every $x\neq0$, and

\begin{equation}
\beta\Delta_p(x,y)\geq\Delta_q(x,y).
\end{equation}

\noindent Therefore, by Lemma \ref{lemma characterization difference norms triangular defect},

\begin{equation}
\beta p-q\in\mathcal{N}(X).
\end{equation}

\noindent Similarly, choose $0<\alpha<\min \{l,a\}$. Then

\begin{equation}
q(x)-\alpha p(x)\geq(l-\alpha)p(x)>0
\end{equation}

\noindent for every $x\neq0$, and

\begin{equation}
\Delta_q(x,y)\geq\alpha\Delta_p(x,y).
\end{equation}

\noindent Hence, again by Lemma \ref{lemma characterization difference norms triangular defect},

\begin{equation}
q-\alpha p\in\mathcal{N}(X).
\end{equation}

\noindent Consequently, $p$ and $q$ belong to the same intrinsic part of $\mathcal{N}(X)$.
\end{proof}

\begin{corollary}\label{corollary intrinsic parts same triangular support}
Let $X$ be a Banach space and $p,q\in\mathcal{N}(X)$. If $p$ and $q$ belong to the same intrinsic part of $\mathcal{N}(X)$, then their triangular defects have the same zero set; that is, for any $x,y\in X$

\begin{equation}
\Delta_p(x,y)=0
\end{equation}

\noindent if and only if

\begin{equation}
\Delta_q(x,y)=0
\end{equation}
\end{corollary}

The converse of Corollary \ref{corollary intrinsic parts same triangular support} does not hold, even in finite-dimensional spaces. In particular, two equivalent norms may have triangular defects with the same zero set while belonging to different intrinsic parts of $\mathcal{N}(X)$, as shown by the following example.

\begin{example}\label{example same zero set different intrinsic parts}
Consider $X = \mathbb{R}^2$ endowed with the equivalent norms

\begin{equation}
p(a,b)=\|(a,b)\|_2= \sqrt{a^2 + b^2}, \qquad  q(a,b)=\|(a,b)\|_4= \sqrt[4]{a^4 + b^4}.
\end{equation}

\noindent Since both $p$ and $q$ are strictly convex, Item $i)$ of Remark \ref{remark properties of triangular defect} implies that their triangular defects have the same zero set, consisting of pairs of vectors lying in the same positive ray. We show, however, that $p$ and $q$ do not belong to the same intrinsic part of $\mathcal{N}(\mathbb{R}^2)$.

Let $x=(1,0)$ and, for every $t>0$, let $x_t=(1,t)$. Since $x$ and $x_t$ are not collinear, strict convexity implies

\begin{equation}
\begin{split}
\Delta_p(x,x_t)  & = p(x)+p(x_t)-p(x+x_t) \\
                 & = 1+\sqrt{1+t^2}-\sqrt{4+t^2} \\
                 & > 0
\end{split}
\end{equation}

\noindent and 

\begin{equation}
\begin{split}
\Delta_q(x,x_t) & = q(x)+q(x_t)-q(x+x_t) \\
                & = 1+ \sqrt[4]{1+t^4}- \sqrt[4]{16+t^4}\\
                & >0.
\end{split}
\end{equation}

\noindent Then, using the Taylor expansion $(1+t)^\alpha=\sum_{n=0}^{\infty}\binom{\alpha}{n}t^n$ around $t=0$, we obtain

\begin{equation}
\Delta_p(x,x_t)=\frac{1}{4}t^2+O(t^4)
\end{equation}

\noindent and

\begin{equation}
\Delta_q(x,x_t)=\frac{7}{32}t^4+O(t^8).
\end{equation}

\noindent Therefore,

\begin{equation}
\lim_{t\to 0^+}\frac{\Delta_q(x,x_t)}{\Delta_p(x,x_t)}=0.
\end{equation}

\noindent Hence there is no constant $a>0$ such that

\begin{equation}
a\Delta_p(x,y)\leq\Delta_q(x,y)
\end{equation}

\noindent for every $x,y\in\mathbb{R}^2$. By Theorem \ref{theorem intrinsic parts triangular defect}, $p$ and $q$ do not belong to the same intrinsic part of $\mathcal{N}(\mathbb{R}^2)$. Thus, even in finite dimension, two equivalent norms may have triangular defects with the same zero set while belonging to different intrinsic parts.
\end{example}

Although equality of the zero sets of the triangular defects is not sufficient for two norms to belong to the same intrinsic part of $\mathcal N(X)$, in finite dimensions it is enough to compare the triangular defects locally near their common zero set, as shown in the following proposition.

\begin{proposition}\label{proposition finite dimensional characterization intrinsic parts}
Let $X$ be a finite-dimensional Banach space, $p,q\in\mathcal{N}(X)$ and

\begin{equation}
D=\{(x,y)\in X^2\mid p(x)+p(y)=1\}.
\end{equation}

\noindent Then the following statements are equivalent:

\begin{itemize}
\item[1)] The norms $p$ and $q$ belong to the same intrinsic part of $\mathcal{N}(X)$.

\item[2)] $\Delta_p$ and $\Delta_q$ have the same zero set $Z$, and for every $z\in Z \cap D$ there exists a relative neighborhood $U_z$ of $z$ in $D$ and constants $a_z, b_z > 0 $ such that

\begin{equation}
a_z \Delta_p (x , y) \leq \Delta_q (x,y) \leq b_z \Delta_p(x,y)
\end{equation}

\noindent for every $(x , y) \in U_z$.
\end{itemize}
\end{proposition}

\begin{proof}
$1)$ implies $2)$ immediately from Theorem \ref{theorem intrinsic parts triangular defect} and Corollary \ref{corollary intrinsic parts same triangular support}.

We prove the converse. By Proposition \ref{Proposition properties of triangular defect}, the triangular defects are positively homogeneous. Therefore, it is enough to study them on the normalized set $D$. Since $X$ is finite-dimensional, $D$ is compact. Moreover, triangular defects are continuous; thus $Z$ is closed and

\begin{equation}
Z'=Z\cap D
\end{equation}

\noindent is a compact subset of $D$. By hypothesis, for every $z\in Z'$ there exist a relative neighborhood $U_z$ of $z$ in $D$ and constants $a_z,b_z>0$ such that

\begin{equation}
a_z \Delta_p (x , y) \leq \Delta_q (x,y) \leq b_z \Delta_p(x,y)
\end{equation}

\noindent for every $(x,y)\in U_z$. By the compactness of $Z'$, finitely many of these neighborhoods cover $Z'$. Hence there exist constants $a_0,b_0>0$ and a relative neighborhood $U$ of $Z'$ in $D$ such that

\begin{equation}
a_0 \Delta_p (x , y) \leq \Delta_q (x,y) \leq b_0 \Delta_p(x,y)
\end{equation}

\noindent for every $(x,y)\in U$.

The set $K = D\setminus U$ is compact and disjoint from $Z'$. Define 

\begin{equation}
R(x,y) = \frac{\Delta_q(x,y)}{\Delta_p(x,y)}
\end{equation}

\noindent for every $(x,y) \in K$. Since $R$ is continuous and positive on $K$, there exist $a_1,b_1>0$ such that

\begin{equation}
a_1\leq R(x,y)\leq b_1
\end{equation}

\noindent for every $(x,y)\in K$. Setting $a=\min\{a_0,a_1\}$ and $b=\max\{b_0,b_1\}$, we obtain

\begin{equation}
a\Delta_p(x,y)\leq\Delta_q(x,y)\leq b\Delta_p(x,y)
\end{equation}

\noindent for every $(x,y)\in D$. By positive homogeneity of the triangular defects, the same inequalities hold for every $(x,y)\in X^2$. Therefore, by Theorem \ref{theorem intrinsic parts triangular defect}, $p$ and $q$ belong to the same intrinsic part of $\mathcal{N}(X)$.
\end{proof}

\section{Intrinsic Hilbert metric on $\mathcal{N}'(X)$}

In this section, we study the intrinsic Hilbert metric on the projective cone of equivalent norms $\mathcal{N}'(X)$. For this purpose, we first introduce the relevant comparison values. Let $X$ be a Banach space and $p,q\in\mathcal{N}(X)$ belong to the same part of $\mathcal{N}(X)$. We define the intrinsic comparison values

\begin{equation}
M_{\mathcal{N}(X)}(q/p)=\inf \{\lambda>0\mid \lambda p-q\in\mathcal{N}(X)\}
\end{equation}

\noindent and

\begin{equation}
m_{\mathcal{N}(X)}(q/p)=\sup\{\lambda>0\mid q-\lambda p\in\mathcal{N}(X)\}.
\end{equation}

\noindent For the triangular defects, define analogously

\begin{equation}
M_{\log}(\Delta_q/\Delta_p) = \inf\left\{\lambda>0\mid\lambda\Delta_p(x,y)\geq\Delta_q(x,y)\text{ for every }x,y\in X \right\}
\end{equation}

\noindent and

\begin{equation}
m_{\log}(\Delta_q/\Delta_p) =\sup\left\{\lambda>0\mid\Delta_q(x,y)\geq\lambda\Delta_p(x,y)\text{ for every }x,y\in X \right\}.
\end{equation}

\noindent By Theorem \ref{theorem intrinsic parts triangular defect} these quantities satisfy

\begin{equation}
0<m_{\log}(\Delta_q/\Delta_p)\leq M_{\log}(\Delta_q/\Delta_p)<\infty.
\end{equation}

In the following proposition, we relate the intrinsic comparison values in $\mathcal{N}(X)$ to the comparison values associated with the symmetric logarithmic distance and the triangular defects.

\begin{proposition}\label{proposition intrinsic comparison values explicit formula}
Let $X$ be a Banach space and $p,q\in\mathcal{N}(X)$ in the same part of $N(X)$. Then

\begin{equation}
M_{\mathcal{N}(X)}(q/p) = \max\{M_{\log}(q/p),M_{\log}(\Delta_q/\Delta_p)\}
\end{equation}

\noindent and

\begin{equation}
m_{\mathcal{N}(X)}(q/p)=\min\{m_{\log}(q/p),m_{\log}(\Delta_q/\Delta_p)\}.
\end{equation}
\end{proposition}
\begin{proof}
Let $p,q\in\mathcal{N}(X)$ be comparable with respect to the intrinsic order of $\mathcal{N}(X)$. We say that $p$ and $q$ satisfy property $P_\lambda$, for $\lambda>0$, if there exists $\varepsilon>0$ such that

\begin{equation}
(\lambda-\varepsilon)p(x)\geq q(x)
\end{equation}

\noindent for every $x\in X$. By Lemma \ref{lemma characterization difference norms triangular defect},

\begin{equation}
\begin{aligned}
M_{\mathcal{N}(X)}(q/p)
&=\inf\left\{\lambda>0\mid p\text{ and }q\text{ satisfy }P_\lambda,\ 
\lambda\Delta_p(x,y)\geq\Delta_q(x,y)\text{ for every }x,y\in X\right\}\\
&=\max\left\{\inf\{\lambda>0\mid p\text{ and }q\text{ satisfy }P_\lambda\},
M_{\log}(\Delta_q/\Delta_p)\right\}.
\end{aligned}
\end{equation}

\noindent We claim that

\begin{equation}
\inf \{\lambda>0\mid p\text{ and }q\text{ satisfy }P_\lambda\}= M_{\log}(q/p).
\end{equation}

\noindent Indeed, property $P_\lambda$ implies $\lambda p(x)\geq q(x)$ for every $x \in X$. Thus 

\begin{equation}
\inf \{\lambda>0\mid p\text{ and }q\text{ satisfy }P_\lambda\} \geq M_{\log}(q/p).
\end{equation}

\noindent On the other hand, if $\lambda>M_{\log}(q/p)$ one may choose $\varepsilon>0$ sufficiently small so that $(\lambda-\varepsilon)p(x)\geq q(x)$ for every $x \in X$. Hence

\begin{equation}
\inf \{\lambda>0\mid p\text{ and }q\text{ satisfy }P_\lambda\} \leq M_{\log}(q/p).
\end{equation}

\noindent Then

\begin{equation}
\inf\{\lambda>0\mid p\text{ and }q\text{ satisfy }P_\lambda\} = M_{\log}(q/p).
\end{equation}

\noindent Consequently,

\begin{equation}
M_{\mathcal{N}(X)}(q/p)= \max\{M_{\log}(q/p),M_{\log}(\Delta_q/\Delta_p)\}.
\end{equation}

\noindent By an analogous argument,

\begin{equation}
m_{\mathcal{N}(X)}(q/p)=\min\{m_{\log}(q/p),m_{\log}(\Delta_q/\Delta_p)\}.
\end{equation}
\end{proof}

\begin{theorem}\label{theorem explicit intrinsic Hilbert metric}
Let $X$ be a Banach space and let $p,q\in\mathcal{N}(X)$ belong to the same intrinsic part. Then the intrinsic Hilbert metric induced by $\mathcal{N}(X)$ satisfies

\begin{equation}
d_{\mathcal{N}(X)}([p],[q]) = \log \frac{ \max \{ M_{\log}(q/p), M_{\log}(\Delta_q/\Delta_p)\}}{\min \{ m_{\log}(q/p), m_{\log}(\Delta_q/\Delta_p) \}}.
\end{equation}
\end{theorem}
\begin{proof}
By the definition of the Hilbert projective metric, we have

\begin{equation}
d_{\mathcal{N}(X)}([p],[q]) = \log \frac{M_{\mathcal{N}(X)}(q/p)}{m_{\mathcal{N}(X)}(q/p)}.
\end{equation}

\noindent Thus, the conclusion follows from Proposition \ref{proposition intrinsic comparison values explicit formula}.
\end{proof}

\begin{remark}
The previous theorem shows that the intrinsic Hilbert metric contains two distinct comparison mechanisms. The symmetric logarithmic quantities $M_{\log}(q/p)$ and $m_{\log}(q/p)$ measure the pointwise comparison between the norms $p$ and $q$, whereas $M_{\log}(\Delta_q/\Delta_p)$ and $m_{\log}(\Delta_q/\Delta_p)$ measure the corresponding comparison between their triangular defects. Thus, the intrinsic metric simultaneously records the relative size of the norms and the way in which they satisfy the triangle inequality. The appearance of the maximum and minimum in the formula reflects the interaction of these two orders. In the next section, we show that this decomposition admits a natural realization as a Hilbert metric on a product cone of nonnegative functions.
\end{remark}

\begin{corollary}\label{corollary comparison of intrinsic and log metrics}
Let $X$ be a Banach space and let $p,q\in\mathcal{N}(X)$ belong to the same intrinsic part. Then
\begin{equation}
d_{\mathcal{N}(X)}([p],[q])\geq d_{\log}([p],[q]).
\end{equation}
\end{corollary}

\begin{remark}
Since

\begin{equation}
\left(\mathcal{N}(X)\cup\{0\}\right)\subset\mathbb{R}^X_+,
\end{equation}

\noindent Propositions \ref{proposition pointwise cone induces logarithmic metric} and \ref{Proposition comparison of Hilbert seminorms wrt subsets} already imply that

\begin{equation}
d_{\mathcal{N}(X)}([p],[q])\geq d_{\log}([p],[q])
\end{equation}

\noindent whenever $p$ and $q$ belong to the same intrinsic part. The significance of Corollary \ref{corollary comparison of intrinsic and log metrics} is that the explicit formula for $d_{\mathcal{N}(X)}$ shows analytically how this refinement arises from the additional comparison of the triangular defects.
\end{remark}

\begin{example}
We now present an explicit example for which the intrinsic Hilbert distance on $\mathcal{N}(X)$ is strictly greater than the symmetric logarithmic distance.

Let $X=\mathbb{R}^2$, let $a>1$, and consider the norms $p,q\in\mathcal{N}(X)$ given by

\begin{equation}
p(x,y)=\sqrt{x^2 + y^2}, \qquad q(x,y)= \sqrt{a x^2 + y^2}.
\end{equation}

\noindent We first compute their symmetric logarithmic distance. Since $p$ and $q$ are equivalent norms, there exist optimal constants $l,u>0$ such that $lp(x,y)\leq q(x,y)\leq up(x,y)$ for every $(x,y)\in\mathbb{R}^2$. The first inequality is equivalent to

\begin{equation}
l^2\left(x^2 + y^2\right)\leq ax^2 + y^2,
\end{equation}

\noindent and therefore

\begin{equation}
(a-l^2) x^2+(1-l^2) y^2 \geq 0.
\end{equation}

\noindent Since $(x,y) \in \mathbb{R}^2$ is arbitrary and $l$ is optimal, it follows that $l=\min\{\sqrt{a},1\}=1$. Analogously, $u=\max\{\sqrt{a},1\}=\sqrt{a}$. Hence,

\begin{equation}\label{technical log simmetric distance}
d_{\log}([p],[q])=\log\frac{\sqrt{a}}{1}=\frac{1}{2}\log a
\end{equation}

\noindent and

\begin{equation}\label{technical log simmetric bounds}
m_{\log} (q/ p) = 1, \qquad M_{\log}(q / p) = \sqrt{a}.
\end{equation}
 
We next obtain two candidate bounds for the comparison of the corresponding triangular defects. Let $x=(1,0)$ and $x_t=(1,t)$ for every $t > 0$. Then

\begin{equation}
\Delta_p(x,x_t)=1+\sqrt{1+t^2}-\sqrt{4+t^2}
\end{equation}

\noindent and

\begin{equation}
\Delta_q(x,x_t)=\sqrt{a}+\sqrt{a+t^2}-\sqrt{4a+t^2}.
\end{equation}

\noindent Using the Taylor expansion $(1+t)^\alpha=\sum_{n=0}^{\infty}\binom{\alpha}{n}t^n$ around $t=0$, we obtain

\begin{equation}
\Delta_p(x,x_t)=\frac{1}{4}t^2+O(t^4), \qquad \Delta_q(x,x_t)=\frac{1}{4\sqrt{a}}t^2+O(t^4).
\end{equation}

\noindent Consequently,

\begin{equation}\label{technical limit 1}
\lim_{t\to0^+}\frac{\Delta_q(x,x_t)}{\Delta_p(x,x_t)}=\frac{1}{\sqrt{a}}.
\end{equation}

\noindent Now let $y=(0,1)$ and $y_t=(t,1)$ for every $t> 0$. Analogously,

\begin{equation}
\Delta_p(y,y_t)=\frac{1}{4}t^2+O(t^4), \qquad \Delta_q(y,y_t)=\frac{a}{4}t^2+O(t^4),
\end{equation}

\noindent and hence

\begin{equation}\label{technical limit 2}
\lim_{t\to0^+}\frac{\Delta_q(y,y_t)}{\Delta_p(y,y_t)}=a.
\end{equation}

\noindent These two limits suggest the optimal comparison constants for the triangular defects. 

We shall now prove that they are indeed universal and optimal, namely,

\begin{equation}
\frac{1}{\sqrt{a}}\Delta_p(x,y)\leq\Delta_q(x,y)\leq a\Delta_p(x,y)
\end{equation}

\noindent for every $x,y\in\mathbb{R}^2$. Since $p$ and $q$ are strictly convex, their triangular defects have the same zero set. Thus, it suffices to consider pairs $(x,y) \in \mathbb{R}^2 \times \mathbb{R}^2$ such that $\Delta_p(x,y)>0$. Let

\begin{equation}
Q_a=
\begin{pmatrix}
a&0\\
0&1
\end{pmatrix}.
\end{equation}

\noindent Then $q(x)^2=x^TQ_ax$, and hence $q$ is a Hilbert norm. If $x=(x_1,x_2)$ and $y=(y_1,y_2)$, its associated inner product is given by $\langle x,y\rangle_q=ax_1y_1+x_2y_2$.

For every Hilbert norm $\| \cdot \|$, rationalizing the triangular defect gives

\begin{equation}
\begin{split}
\|x\|+\|y\|-\|x+y\| &=\frac{(\|x\|+\|y\|)^2-\|x+y\|^2}{\|x\|+\|y\|+\|x+y\|}\\
              &=\frac{2\|x\|\|y\|-2\langle x,y\rangle}{\|x\|+\|y\|+\|x+y\|}.
\end{split}
\end{equation}

\noindent Therefore,

\begin{equation}
\frac{\Delta_q(x,y)}{\Delta_p(x,y)}
=
\frac{p(x)+p(y)+p(x+y)}{q(x)+q(y)+q(x+y)}
\frac{q(x)q(y)-\langle x,y\rangle_q}{p(x)p(y)-\langle x,y\rangle_p}.
\end{equation}

\noindent Set $A=p(x)+p(y)+p(x+y)$, $B=q(x)+q(y)+q(x+y)$, $C=q(x)q(y)-\langle x,y\rangle_q$, and $D=p(x)p(y)-\langle x,y\rangle_p$. We estimate $A/B$ and $C/D$ separately.

Since $a>1$, for every $x\in\mathbb{R}^2$ we have $p(x)\leq q(x)\leq\sqrt{a} p(x)$. Consequently,

\begin{equation}
A\leq B\leq\sqrt{a}A,
\end{equation}

\noindent and therefore

\begin{equation}
\frac{1}{\sqrt{a}}\leq\frac{A}{B}\leq1.
\end{equation}

\noindent We now estimate $C/D$. First, we prove that $C/D\geq1$. This is equivalent to

\begin{equation}
q(x)q(y)\geq p(x)p(y)+(a-1)x_1y_1.
\end{equation}

\noindent If the right-hand side is nonpositive, the inequality is immediate. Suppose that it is nonnegative. Squaring both sides, the desired inequality is equivalent to

\begin{equation}
(ax_1^2+x_2^2)(ay_1^2+y_2^2)\geq\left(p(x)p(y)+(a-1)x_1y_1\right)^2.
\end{equation}

\noindent Expanding and simplifying yields

\begin{equation}
2x_1^2y_1^2+x_1^2y_2^2+x_2^2y_1^2\geq2p(x)p(y)x_1y_1.
\end{equation}

\noindent Since

\begin{equation}
2x_1^2y_1^2+x_1^2y_2^2+x_2^2y_1^2=x_1^2p(y)^2+y_1^2p(x)^2,
\end{equation}

\noindent the preceding inequality becomes

\begin{equation}
\left(x_1p(y)-y_1p(x)\right)^2\geq0.
\end{equation}

\noindent Hence $C/D\geq1$.

We now prove that $C/D\leq a$. This inequality is equivalent to

\begin{equation}
q(x)q(y)+(a-1)x_2y_2\leq ap(x)p(y).
\end{equation}

\noindent If the left-hand side is nonpositive, the inequality follows immediately. Suppose therefore that it is positive. Squaring both sides and simplifying, we obtain

\begin{equation}
x_2^2q(y)^2+y_2^2q(x)^2-2q(x)q(y)x_2y_2\geq0.
\end{equation}

\noindent Equivalently,

\begin{equation}
\left(x_2q(y)-y_2q(x)\right)^2\geq0.
\end{equation}

\noindent Hence $C/D\leq a$. Therefore,

\begin{equation}
1\leq\frac{C}{D}\leq a.
\end{equation}

\noindent Combining the estimates for $A/B$ and $C/D$, we obtain

\begin{equation}
\frac{1}{\sqrt{a}}\leq\frac{\Delta_q(x,y)}{\Delta_p(x,y)}\leq a.
\end{equation}

\noindent Therefore, $p$ and $q$ belong to the same intrinsic part by Theorem \ref{theorem intrinsic parts triangular defect}. The two limits obtained above, (\ref{technical limit 1}) and (\ref{technical limit 2}), show that both constants $\frac{1}{\sqrt{a}}$ and $a$ are optimal. Hence,

\begin{equation}
m_{\log}(\Delta_q/\Delta_p)=\frac{1}{\sqrt{a}},\qquad M_{\log}(\Delta_q/\Delta_p)=a.
\end{equation}

\noindent By (\ref{technical log simmetric distance}) and (\ref{technical log simmetric bounds}), we have $m_{\log}(q/p)=1$ and $M_{\log}(q/p)=\sqrt{a}$. Therefore, Theorem \ref{theorem explicit intrinsic Hilbert metric} gives the explicit formula for the intrinsic Hilbert metric

\begin{equation}
\begin{split}
d_{\mathcal{N}(X)}([p],[q]) &=\log\frac{\max\{M_{\log}(q/p), M_{\log}(\Delta_q / \Delta_p)\}}{\min\{m_{\log}(q/p), m_{\log}(\Delta_q / \Delta_p)\}} \\
                            &=\log\frac{\max\{\sqrt{a},a\}}{\min\{1,1/\sqrt{a}\}}\\
                            &=\log\frac{a}{1/\sqrt{a}}\\
                            &=\frac{3}{2}\log a\\
                            &=3d_{\log}([p],[q]).
\end{split}
\end{equation}
\end{example}

\section{Canonical representation of intrinsic Hilbert metric on $\mathcal{N}'(X)$}
In this section, we show that the intrinsic Hilbert metric on $\mathcal{N}(X)$ can be realized as a natural Hilbert metric on a cone of nonnegative functions.

Let $A$ be a set. We denote by

\begin{equation}
\mathbb{R}^A_+=\{f:A\longrightarrow\mathbb{R}\mid f(a)\geq0\text{ for every }a\in A\}
\end{equation}

\noindent the cone of nonnegative real-valued functions on $A$.

Let $X$ be a Banach space. Define

\begin{equation}
\Phi:\mathcal{N}(X)\longrightarrow\mathbb{R}^X_+\times\mathbb{R}^{X\times X}_+
\end{equation}

\noindent by

\begin{equation}
\Phi(p)=(p,\Delta_p).
\end{equation}

\noindent The assignment $p\mapsto\Delta_p$ is additive and positively homogeneous, and extends linearly to $\mathcal{N}(X)-\mathcal{N}(X)$. In particular, $\Phi(\lambda p)=\lambda\Phi(p)$ for every $\lambda>0$, and hence $\Phi$ induces a well-defined map on the corresponding projective spaces given by $\widetilde{\Phi}([p])=[\Phi(p)]$.

Recall that if $C$ and $D$ are two cones with vertex, their product $C\times D$ is naturally endowed with the product order defined by

\begin{equation}
(c_1,d_1)\preceq_{C\times D}(c_2,d_2) \quad \Longleftrightarrow \quad c_2-c_1\in C\text{ and }d_2-d_1\in D.
\end{equation}

\noindent In what follows, every product of cones will be understood to be endowed with this order.

Set $K = C \times D$. Notice that if $p=(p_1,p_2) ,q=(q_1,q_2) \in K$ belong to the same intrinsic part, then

\begin{equation}\label{technical canonical decomposition of product intrinsict values 1}
\begin{split}
M_K(q/p) &=\inf\{\lambda>0 \mid q\preceq_K \lambda p\} \\
         &=\inf\{\lambda>0\mid q_1\preceq_C \lambda p_1 \text{ and } q_2\preceq_D \lambda p_2\}\\
         &=\max\{M_C(q_1/p_1),M_D(q_2/p_2)\}.
\end{split}
\end{equation}

\noindent Similarly, 

\begin{equation}\label{technical canonical decomposition of product intrinsict values 2}
m_K(q/p) = \min \{m_C(q_1/p_1),m_D(q_2/p_2) \}.
\end{equation}

The following technical lemma relates the intrinsic comparison values of the triangular defects to the comparison values induced by the pointwise order.

\begin{lemma}\label{lemma pointwise product cone induces logarithmic metric on triangular defects}
Let $X$ be a Banach space, and let $p,q\in\mathcal{N}(X)$ belong to the same intrinsic part. If

\begin{equation}
D=\mathbb{R}^{X\times X}_+,
\end{equation}

\noindent then

\begin{equation}
M_D(\Delta_q/\Delta_p)=M_{\log}(\Delta_q/\Delta_p), \qquad m_D(\Delta_q/\Delta_p)=m_{\log}(\Delta_q/\Delta_p).
\end{equation}
\end{lemma}
\begin{proof}
Let $\lambda>0$. Then 

\begin{equation}
\Delta_q\preceq_D\lambda\Delta_p
\end{equation}

\noindent if and only if $\lambda\Delta_p-\Delta_q\in\mathbb{R}^{X\times X}_+$, which is equivalent to

\begin{equation}
\Delta_q(x,y)\leq\lambda\Delta_p(x,y)
\end{equation}

\noindent for every $x,y\in X$. Therefore,

\begin{equation}
M_D(\Delta_q/\Delta_p)=M_{\log}(\Delta_q/\Delta_p).
\end{equation}

Similarly,

\begin{equation}
\lambda\Delta_p\preceq_D\Delta_q
\end{equation}

\noindent if and only if

\begin{equation}
\lambda\Delta_p(x,y)\leq\Delta_q(x,y)
\end{equation}

\noindent for every $x,y\in X$. Hence,

\begin{equation}
m_D(\Delta_q/\Delta_p)=m_{\log}(\Delta_q/\Delta_p).
\end{equation}
\end{proof}

\begin{proposition}
Let $X$ be a Banach space and let

\begin{equation}
K=\mathbb{R}^X_+\times\mathbb{R}^{X\times X}_+.
\end{equation}

\noindent If $p,q\in\mathcal{N}(X)$ belong to the same intrinsic part, then

\begin{equation}
d_K([\Phi(p)],[\Phi(q)])=d_{\mathcal{N}(X)}([p],[q]).
\end{equation}
\end{proposition}

\begin{proof}
Set $C= \mathbb{R}^X_+$ and $D= \mathbb{R}^{X \times X}_+$. Since $K=C\times D$ is endowed with the product order, using (\ref{technical canonical decomposition of product intrinsict values 1}), (\ref{technical canonical decomposition of product intrinsict values 2}), Proposition \ref{proposition pointwise cone induces logarithmic metric} and Lemma \ref{lemma pointwise product cone induces logarithmic metric on triangular defects} we have

\begin{equation}
\begin{split}
M_K(\Phi(q)/\Phi(p)) & = \max\left\{M_C(q/p),M_D(\Delta_q/\Delta_p)\right\} \\
                     & = \max\left\{M_{\log}(q/p),M_{\log}(\Delta_q/\Delta_p)\right\}.
\end{split}
\end{equation}

\noindent and

\begin{equation}
\begin{split}
m_K(\Phi(q)/\Phi(p)) & = \min \left\{m_C(q/p),m_D(\Delta_q/\Delta_p)\right\} \\ 
                     & = \min \left\{m_{\log}(q/p),m_{\log}(\Delta_q/\Delta_p)\right\}.
\end{split}
\end{equation}

\noindent The conclusion now follows from Theorem \ref{theorem explicit intrinsic Hilbert metric}.
\end{proof}

Although the preceding proposition realizes the intrinsic metric through the cone $K$, the order induced by $K$ on $\Phi(\mathcal{N}(X))$ does not in general coincide with the intrinsic order of $\mathcal{N}(X)$. Indeed, pointwise nonnegativity of the first coordinate does not guarantee that a difference of equivalent norms is itself an equivalent norm. We next restrict the first coordinate in order to recover the intrinsic cone order exactly.

We define

\begin{equation}
U(X) = \left\{
f\in\mathcal{N}(X)-\mathcal{N}(X)\;\middle|\;
\begin{array}{c}
\text{there exists }r\in\mathcal{N}(X) \text{ such that}\\[1mm]
f(x)\geq r(x)\text{ for every }x\in X
\end{array}
\right\} \cup \{ 0 \}
\end{equation}

\noindent Note that 

\begin{equation}
\mathcal{N}(X) \subset U(X) \subset \mathbb{R}^X_+.
\end{equation}

\begin{proposition}
Let $X$ be a Banach space. Then $U(X)$ is a cone with vertex.
\end{proposition}
\begin{proof}
Let $f_1,f_2\in U(X)\setminus \{0\}$ and $\lambda>0$. Then, for each $i=1,2$, there exist $p_i,q_i,r_i\in\mathcal{N}(X)$ such that

\begin{equation}
f_i(x)=p_i(x)-q_i(x)\geq r_i(x)
\end{equation}

\noindent for every $x\in X$. Thus

\begin{equation}
\lambda f_1 (x)  = \lambda  p_1(x) - \lambda q_1 (x) \geq \lambda r_1(x), \qquad  x \in X 
\end{equation}

\noindent and

\begin{equation}
f_1 (x) +  f_2 (x) = (p_1 + p_2) (x) - (q_1 + q_2)(x) \geq (r_1 + r_2) (x), \qquad x \in X.
\end{equation}

\noindent That is, $\lambda f_1 , f_1 + f_2 \in U(X)$. 

Finally, let $f \in U(X) \cap (-U(X))$. Thus $f, -f \in U(X)$. If $f \neq 0$, then there exists $r, s\in \mathcal{N}(X)$ such that $f(x) \geq r(x)$ and $-f(x) \geq s(x)$ for every $x \in X$. Combining these two inequalities yields a contradiction. Hence $f \equiv 0$.
\end{proof}

Now we define 

\begin{equation}
H(X) = U(X) \times \mathbb{R}^{X\times X}_+.
\end{equation}

\noindent The cone $H(X)$ can be viewed as a canonical cone of nonnegative functions associated with $\mathcal{N}'(X)$. 

\begin{theorem}
Let $X$ be a Banach space. For every $p,q\in\mathcal{N}(X)$,

\begin{equation}
\Phi(q)\preceq_{H(X)}\Phi(p) \quad\Longleftrightarrow\quad q\preceq_{\mathcal{N}(X)} p.
\end{equation}

\noindent Moreover, if $p$ and $q$ belong to the same intrinsic part, then

\begin{equation}
d_{H(X)}([\Phi(p)],[\Phi(q)])=d_{\mathcal{N}(X)}([p],[q]).
\end{equation}
\end{theorem}
\begin{proof}
By definition of the product order,

\begin{equation}
\Phi(q)\preceq_{H(X)}\Phi(p)
\end{equation}

\noindent if and only if $p-q \equiv 0$ or $p - q \in U(X)$ and $\Delta_p - \Delta_q \in \mathbb{R}^{X \times X}_+$. In the second case, there exists $r\in\mathcal{N}(X)$ such that

\begin{equation}\label{technical difference of norms}
p(x)-q(x)\geq r(x)
\end{equation}

\noindent for every $x\in X$, and

\begin{equation}\label{technical difference of triangular defects}
\Delta_p(x,y)-\Delta_q(x,y)\geq0
\end{equation}

\noindent for every $x,y\in X$. Since $r$ and $p$ are equivalent norms, there exists $\varepsilon >0$ such that $r(x) \geq \varepsilon p(x)$ for every $x \in X$. Thus, using (\ref{technical difference of norms}), we have

\begin{equation}\label{technical difference of norms 2}
p(x) - q(x) \geq \varepsilon p(x), \qquad x \in X
\end{equation}  

\noindent By Lemma \ref{lemma characterization difference norms triangular defect}, in the second case, conditions (\ref{technical difference of triangular defects}) and (\ref{technical difference of norms 2}) are equivalent to $p-q\in\mathcal{N}(X)$. Together with the case $p -q \equiv 0$, we obtain

\begin{equation}\label{technical same cone order}
\Phi(q)\preceq_{H(X)}\Phi(p) \quad \Longleftrightarrow \quad q \preceq_{\mathcal{N}(X)} p.
\end{equation}

We now prove the isometric identity. By the linearity of $\Phi$ and (\ref{technical same cone order}), we have

\begin{equation}
\begin{split}
M_{H(X)}(\Phi(q)/\Phi(p)) & = \inf\left\{\lambda>0\mid \Phi(q)\preceq_{H(X)}\lambda\Phi(p)\right\} \\
                     & = \inf\left\{\lambda>0\mid q\preceq_{\mathcal{N}(X)}\lambda p\right\}\\
                     & = M_{\mathcal{N}(X)}(q/p).
\end{split}
\end{equation}

\noindent Analogously,

\begin{equation}
m_{H(X)}(\Phi(q)/\Phi(p)) = m_{\mathcal{N}(X)}(q/p).
\end{equation}

\noindent Finally,

\begin{equation}
d_{H(X)}([\Phi(p)],[\Phi(q)]) = d_{\mathcal{N}(X)}([p],[q]).
\end{equation}
\end{proof}

\begin{remark}
The preceding theorem shows that the triangular defect is not merely an auxiliary quantity used to characterize differences of equivalent norms. The map $\Phi$, given by $\Phi(p)=(p,\Delta_p)$, represents the intrinsic order of $\mathcal{N}(X)$ as an order on a cone of nonnegative functions and preserves the corresponding Hilbert metric on each intrinsic part. In this sense, the pair $(p,\Delta_p)$ provides two natural coordinates for the intrinsic cone geometry of $\mathcal{N}(X)$: the first records the pointwise size of the norm, while the second records its triangular structure.

\end{remark}

\end{document}